\def\draftdate{\today}

\documentclass[hidelinks]{amsart}

\usepackage{stmaryrd, verbatim}
\usepackage{amssymb}
\usepackage{mathrsfs}
\usepackage{cancel}
\usepackage{wrapfig}
\usepackage{tikz}
\usepackage{xy}
\usepackage{hyperref}

\usetikzlibrary{calc,fadings,decorations.pathreplacing}

\xyoption{arrow}
\xyoption{curve}
\xyoption{matrix}
\xyoption{cmtip}
\SelectTips{cm}{}
\DeclareMathOperator{\haar}{haar}
\DeclareMathOperator{\Lip}{Lip}
\DeclareMathOperator{\net}{net}

\mathchardef\varDelta="7101

\newcommand{\bR}{{\mathbb{R}}}

\def\quickop#1{\expandafter\DeclareMathOperator\csname
#1\endcsname{#1}}
\quickop{id}\quickop{Id}\quickop{op}
\quickop{Hom}
\quickop{Ob}\quickop{diag}
\quickop{hocolim}
\quickop{Sd}\quickop{Ex}

\numberwithin{equation}{section}

\newtheorem{thm}[equation]{Theorem}
\newtheorem*{thm*}{Theorem}
\newtheorem{cor}[equation]{Corollary}
\newtheorem{lem}[equation]{Lemma}

\theoremstyle{definition}
\newtheorem{defn}[equation]{Definition}

\theoremstyle{remark}
\newtheorem{rem}[equation]{Remark}

\newdir{ >}{{}*!/-5pt/\dir{>}}

\begin{document}

\title%
[Generating an equidistributed net on a sphere]{Generating an equidistributed net on a unit $n$-sphere using random rotations}

\author{Somnath Chakraborty}
\author{Hariharan Narayanan}
\address{School of Technology and Computer Science, Tata Institute of Fundamental Research,
Mumbai 400005, India}
\email{\href{chakraso@indiana.edu}{chakraso@indiana.edu}}
\email{\href{hariharan.narayanan@tifr.res.in}{hariharan.narayanan@tifr.res.in}}

\date{\draftdate}

\begin{abstract}
We develop a randomized algorithm (that succeeds with high probability) for 
generating an $\epsilon$-net in a sphere of dimension $n$. The 
basic scheme is to pick $O(n \ln (1/\epsilon) + \ln (1/\delta))$ random rotations and take 
all possible words of length ${ O}(n\ln (1/\epsilon))$ in the same alphabet and 
act them on a fixed point. We show this set of points is equidistributed at 
a scale of $\epsilon$. Our main application is to approximate integration 
of Lipschitz functions over an $n$-sphere.\begin{comment} The net we produce, is 
$\epsilon$-close in Hausdorff distance to the $n$-sphere, and it is also 
equidistributed (with high probability) in the following sense: the uniform 
counting measure $\mu_{\mbox{net}}$ over the net is $\epsilon$-close to the uniform 
measure on the $n-$sphere, $\mu_{\mbox{haar}}$, in  $1$-Wasserstein distance. 
This is equivalent to the statement that the integral of every $1$-Lipschitz function on $S^n$ with 
respect to $\mu_{\mbox{net}}$ is $\epsilon$-close to its integral with respect to $\mu_{\mbox{haar}}$.\end{comment}

\begin{comment}
We show that, for every sufficiently small $\epsilon>0$, an equidistributed 
$\epsilon$-net on a symmetric space 
can be obtained under ${\tilde O}(\log 1/\epsilon)$ degree of randomness.
\end{comment}
\end{abstract}

\maketitle

\section{Introduction}

In the present article, we develop a randomized algorithm (with high success probability) for the 
generation of an $\epsilon$-net in an unit sphere of dimension $n$. The basic scheme is 
to pick $k := {O}(n\ln (1/\epsilon) + \ln (1/\delta))$ random rotations and take all possible words on length 
$l := {O}(n\ln (1/\epsilon))$ in the same alphabet and act them on a fixed point. We show 
this set  of points is equidistributed at a scale of $\epsilon$. The group $SU(2)$ 
can be identified with the three dimensional sphere $S^3$, thus we obtain a 
scheme for producing an $\epsilon$-net of $SU(2)$, a task which is relevant to 
quantum computing in the context of the Solovay-Kitaev algorithm (although we 
do not address the issue of generating a sequence of elementary gates to efficiently 
approximate any given gate, in a non-exhaustive fashion). Our main application is to 
integration of Lipschitz functions over an $n$-sphere. The net we produce, is 
$\epsilon$-close in Hausdorff distance to the $n$-sphere, and is also equidistributed 
in the following sense (with high probability): The uniform counting  measure 
$\mu_{\net}$ over the net is close to the uniform measure on the $n$-sphere 
$\mu_{\haar}$ in  $1$-Wasserstein distance. This implies that the integral of 
every $1$-Lipschitz function on $S^n$ with respect to $\mu_{\net}$ is $\epsilon$-close 
to its integral with respect to $\mu_{\haar}$.\\

In \cite{AR}, Alon and Roichman proved that, \emph{given any $\delta>0$, there exists a 
$c(\delta)>0$ such that for any finite group $G$, and a random subset $S\subset G$ 
of order at least $c(\delta)\log|G|$, the induced Cayley graph $\chi(G,S)$ 
has small normalized second largest eigenvalue (in absolute value)}: 
\begin{equation}\label{eqnD}
\mathbb E\begin{bmatrix}|\lambda_2^\ast(\chi(G,S))|\end{bmatrix}<\delta.
\end{equation} Considering random walk on expander multigraphs, it follows that every 
element $g\in G$ is an $S$-word of length at most $\log |G|$. 
For an irreducible representation $\rho\in {\hat G}$, let $d_\rho$ 
be its dimension; let $R$ be the regular representation of $G$, and 
$D=\sum_{\rho\in{\hat G}}d_\rho$. In \cite{LRuss}, Russel and Landau 
proved that (\ref{eqnD}) holds for all random subsets 
$S\subset G$ of order at least 
\[\begin{pmatrix}{\frac{2\ln 2}{\epsilon}}+o(1)\end{pmatrix}^2\log |D|\] This 
was obtained via an application of \emph{tail bounds for operator-valued 
random variables}, as in Ahlswede and Winter \cite{AW}, building upon the following observation: 
\emph{the normalized adjacency matrix of $\chi(G,S)$ is the operator \[(2|S|)^{-1}
\sum_{s\in S}(R(s)+R(s^{-1})),\] 
presented in terms of the standard basis of $\mathbb C[G]$}.\\

Now let $G$ be a compact Lie group, and $\mu$ a left-invariant
Borel probability measure on $G$. One considers the averaging 
operator $z_\mu:L^2(G)\rightarrow L^2(G)$, given by 
\[z_\mu(f)(x)=\int_Gf(xg)~d\mu(g)\] In \cite{BG-SUd}, Bourgain and Gamburd established that 
if $G=SU(d)$ then $z_\mu$ has 
spectral radius $<1$ when $\mbox{supp}(\mu)$ is finite algebraic subset generating 
a nonabelian free subgroup of $G$. This has since been extended to all 
compact connected simple Lie groups by Benoist and de Saxc\'{e} in \cite{BdeS}, where it was 
shown that $z_\mu$ has a spectral gap if and only if $\mu$ is \emph{almost 
diophantine}. A corollary to the main result in \cite{BdeS} is the following: if 
$\mu$ is finitely-supported almost diophantine then the set of words in $\mbox{supp}(\mu)$ 
of fixed length approaches $G$ in Hausdorff distance.\\

In \cite{Hari}, a quantitative version of 
the spectral gap question was considered. It was shown that the Hausdorff distance 
between $G$, a compact connected Lie group, and the subset of fixed length 
words on a random essentially small finite alphabet $S\subset G$ decays exponentially 
in the length of the words, with high probability. This was done via an analysis of 
the heat kernel with respect to a suitable finite dimensional subspace of $L^2(G)$ and an application 
of \emph{tail bounds for operator-valued random variables}. 
We note that the results of the present article are not implied by the results of \cite{Hari}, 
because the dimension of the Lie group $SO_n$ is $n(n+1)/2$, and so 
the bounds from \cite{Hari} for the length of the words and the number 
of generators, that apply for general compact Lie groups would be quadratic 
in $n$ rather than linear in $n$. 
In the special case of the unitary group $U_n$, such a result with a 
quadratic dependence on dimension for the length of the words $n$ was 
previously obtained by Hastings and Harrow in Theorem 5 of \cite{Hastings-Harrow}, 
however in their result the number of generators is specified in a indirect manner, 
whose dependence on $n$ is not obvious. On the other hand, the bounds 
obtained in the present work are linear in $n$, both 
for the number of generators and the length of the words. In fact, for these 
parameters, the value of $(2k)^l$ is close to the volumetric lower bound of 
$(1/\epsilon)^{\Omega(n)}$ on the size of an $\epsilon-$net of $S^n$.\\

The two main results of this paper are stated below. The numberings correspond 
to their appearances in Sections $3$ and $4$ respectively. In the following statements, $C_n$ 
denotes a certain positive constant depending on $n$. For a finite set $S$, the $l$-fold 
product $S^l$ consists of all $S$-words of length $l$ inside the free group generated by $S$.\\

%\begin{thm}\label{main1-intro}
\noindent {\bf Theorem 3.15:}\\
\emph{Let $\epsilon\in (0,{\frac 1{3n}})$ and $\delta\in (0,1)$; let $r=2\epsilon
\sqrt{\ln {\frac {3C_n}{\epsilon^{2n-1}}}}$. Let $S\subset SO_{n+1}$ 
consist of $k$ iid random points, drawn from the Haar measure on $SO_{n+1}$, where 
\[k\geq 8\ln 2\begin{pmatrix}(n+4)+2\ln\begin{pmatrix}{\frac 1{\delta}}\end{pmatrix}
+6n(1+a_n)\ln\begin{pmatrix}{\frac 1{\epsilon}}\end{pmatrix}-\ln (n!)\end{pmatrix},\] 
and $a_n:={\frac {2\log_2\log_2(5n)} {\log_2(5n)}}$. Let ${\hat S}:
=S\sqcup S^{-1}$ be the (multi)set of all elements in $S$ and their inverses. Let $l= 
{\frac n2}\log_2\begin{pmatrix}{\frac 1{r\epsilon}}\end{pmatrix}+(4+3a_n)n\log_2 
\begin{pmatrix}{\frac 1{\epsilon}}\end{pmatrix}$; if $r$ is sufficiently small then the probability  that 
$x_0{\hat S}^l\subseteq S^n$ is an $r$-net in $S^n$  is at least $1-\delta$.}\hfill\(\Box\)\\

\noindent {\bf Theorem 4.14:}\\
\emph{For $n>1$, let $\sigma$ be the probability measure on $S^n$ corresponding to haar 
probability measure on the group of rotations $SO_{n+1}$.
Let $\epsilon,\delta>0$ be sufficiently small and $r=2\epsilon
\sqrt{\ln {\frac {3C_n}{\epsilon^{2n-1}}}}$. Let $S\subseteq SO_{n+1}$ be a random 
subset such that $|S|=k$ satisfies %the inequality in Theorem~\ref{main}, namely 
\[k\geq 8\ln 2\begin{pmatrix}(n+4)+2\ln\begin{pmatrix}{\frac 1{\delta}}\end{pmatrix}
+6n(1+a_n)\ln\begin{pmatrix}{\frac 1{\epsilon}}\end{pmatrix}-\ln (n!)\end{pmatrix},\] 
where $a_n:={\frac {2\log_2\log_2(5n)}
{\log_2(5n)}}$. Let ${\hat S}:
=S\sqcup S^{-1}$ be the (multi)set of all elements in $S$ and their inverses. 
Let $x_0\in S^n$ and let $\nu$ be the probability measure on $S^n$, 
uniformly supported on ${\hat S}^lx_0$, where \[l={\frac n2}\log_2
\begin{pmatrix}{\frac 1{r\epsilon}}\end{pmatrix}+(4+3a_n)\log_2
\begin{pmatrix}{\frac 1{\epsilon}}\end{pmatrix}.\] If $r$ is sufficiently small, then the 
following inequality holds  with 
probability at least $1-\delta$: \[W_1(\sigma,\nu)\leq \epsilon,\] 
Here $W_1$ is used to denote the $1-$Wasserstein distance between 
two measures supported on $S^n$.}\hfill\(\Box\)\\
%\hfill\(\Box\)\end{thm}

%\begin{cor}\label{precision-intro}
For the remainder of this section, we assume given a real number model of computation, in which 
only standard algebraic operations are allowed on Gaussian random vectors, 
but bits are not manipulated. 
Thus, for $k$ and $l$ as described in Theorem~\ref{Wassers}, 
choose a set $S$ consisting of $k$ orthogonal matrices, each chosen 
independently from the Haar measure of $SO_{n+1}$. Consider all $(2k)^l$ 
words of length $l$ in these generators and their inverses. Apply the resulting 
matrices to the vector ${\bf e}_{n+1} = (0,0,\cdots,0,1)^T$. Then these $(2k)^l$ points 
form an equidistributed net, that can be used for integrating a $1-$Lipschitz to 
within an additive error of $\epsilon$.
If we assume an oracle that outputs  independent $n-$dimensional Gaussian random 
vectors when queried, then the whole process requires only $kn$ queries to this oracle. 
Note that the obvious procedure of producing an equidistributed net, would require 
$\epsilon^{-\Omega(n)}$ calls to the Gaussian oracle (which would then be normalized 
to lie on the sphere). The latter method uses exponentially more randomness than our 
procedure using random rotations. 
%\end{cor}
\begin{comment}
The results of this paper are not implied by the results of \cite{Hari}, because 
apart from the less explicit control over constants that results from 
the use of heat kernel bounds due to Cheng-Li-Yau from \cite{CLYau}, 
the dimension of the Lie group SO(n) is $n(n+1)/2$, and so the bounds 
from \cite{Hari} that apply for general compact Lie groups would be 
quadratic in $n$ rather than linear in $n$. 
In the special case of the unitary group $U_n$, such a result with a 
quadratic dependence on dimension $n$ was  previously obtained by 
Hastings and Harrow in Theorem 5 of \cite{Hastings-Harrow}. 
On the other hand, the bounds  obtained in the present work are linear in $n$. In 
an ongoing work, similar results are obtained for Riemannian symmetric spaces.\\
In the present article, we show that for a compact connected symmetric space $X=H\backslash G$, 
where $G$ is a connected semisimple Lie group, the distribution corresponding to
the heat-kernel smoothening of a random finitely supported measure on $X$ \emph{efficiently} 
approximates the uniform distribution on $X$. We approach this following \cite{Hari}: 
the averaging operator $z_\mu$ has been restricted to a finite dimensional 
subspace of $L^2_0(G)$, the space of mean-zero square-integrable functions. blah blah blah ...
\end{comment}

\subsection*{Acknowledgments}
SC would like to thank Sandeep Juneja 
and Jaikumar Radhakrishnan for helpful conversations. 
HN was partially supported by a Ramanujan fellowship.

\section{Application of Spherical Harmonics}
\label{spec-sphere}

%\subsection{Application of Spherical Harmonics}

This section briefly reviews the basics of harmonic analysis on the unit $n$-sphere 
$S^n\subset \mathbb R^{n+1}$. The two lemmas in this section will be used in an 
essential manner in deriving the computations in this section.\\

Let $\sigma$ denote the standard euclidean surface probability measure on $S^n$. 
For any Borel set $B\subset S^n$, if ${\hat B}:=\{\alpha x:x\in B, \alpha\in [0,1]\}$ 
then \[\sigma(B)={\frac{\lambda({\hat B})}{\lambda(D)}},\] where 
$\lambda$ is the standard Lebesgue measure on $\bR^{n+1}$ and 
$D\subseteq \bR^{n+1}$ is the unit disk centered at origin, so that $S^n=\partial D$. We recall 
that the Lebesgue measure of the unit $n$-sphere in $\bR^n$ is 
\[\Omega_n:=\lambda(S^n)={\frac {2\pi^{\frac {n+1}2}}{\Gamma\begin
{pmatrix}\frac {n+1}2\end{pmatrix}}}.\] Now, the fact that $\sigma$ is $SO_{n+1}$-
invariant follows from usual rotation invariance of Lebesgue measure 
on $\bR^{n+1}$. Hence (by unimodularity of the compact Lie group 
$SO_{n+1}$ of rotations of $S^n$), the measure $\sigma$ is the unique 
probability measure on $S^n$ induced by the haar measure on $SO_{n+1}$. 
For $n>1$, let $\Delta:=\Delta_{S^n}$ be the negative of the Laplace-Beltrami 
operator on $S^n$. Thus, given $g\in C^2(S^n)$, one has \[-\Delta (g)=
\Delta_{\bR^n}({\tilde g})\mid_{S^n}\] where ${\tilde g}:\bR^{n+1}-\{0\}\rightarrow \bR$ 
is defined by ${\tilde g}(x)=g(|x|^{-1}x)$. It is well-known that the Hilbert-product 
space $L^2(S^n)$ decomposes into a direct sum of the eigenspaces of 
$\Delta$, in the sense that the $L^2$-closure of the direct sum is 
$L^2(S^n)$: \begin{align}\label{decom} L^2(S^n)=\bigoplus_{k=0}^\infty 
H_k(S^n)\end{align} Recall that $H_k(S^n)$ is the 
space of degree-$k$ homogeneous harmonic polynomials in $n+1$ variables, restricted to $S^n$; 
the dimension of $H_k(S^n)$ is \begin{equation}\label{dimen}h_k:=\begin{pmatrix}n+k\\n\end{pmatrix}
-\begin{pmatrix}n+k-2\\n\end{pmatrix}\end{equation} and the corresponding eigenvalue 
is $\lambda_k:=k(n+k-1)$. Note that, for any $n,k>0$, one has \begin{align*}
\sum_{a=0}^k\dim H_a(S^n)&=\sum_{a=0}^k\begin{pmatrix}\begin{pmatrix}
n+a\\a\end{pmatrix}-\begin{pmatrix}n+a-2\\a-2\end{pmatrix}\end{pmatrix}
\\ &=\begin{pmatrix}n+k\\k\end{pmatrix}+\begin{pmatrix}n+k-1\\k-1\end{pmatrix}
\\ &=\dim H_k(S^{n+1})\end{align*} One has 
\begin{align*}\lim_{k\rightarrow\infty}{\frac {\displaystyle\sum_{k\leq a\leq k\sqrt[n]{2}}
\dim H_a(S^n)}{k^n}}& = 2\lim_{k\rightarrow\infty}
{\frac {\displaystyle\sum_{0\leq a\leq \lfloor k\sqrt[n]{2}\rfloor}
\dim H_a}{2k^{n}}} - \lim_{k\rightarrow\infty}{\frac {\displaystyle\sum_{0\leq a\leq k}\dim H_a}{k^n}}\\ 
& = 2\lim_{l\rightarrow\infty}{\frac {\displaystyle\sum_{0\leq a\leq l}\dim H_a}{l^n}} -
 \lim_{k\rightarrow\infty}{\frac {\displaystyle\sum_{0\leq a\leq k}\dim H_a}{k^n}}\\ & = 
\lim_{k\rightarrow\infty}{\frac {\displaystyle\sum_{0\leq a\leq k}\dim H_a}{k^n}} \\ & =
 \lim_{k\rightarrow\infty}{\frac {\dim H_k(S^{n+1})}{k^n}} \\  & = \lim_{k\rightarrow\infty}k^{-n}
\begin{pmatrix}n+k\\k\end{pmatrix}+\lim_{k\rightarrow\infty}k^{-n}
\begin{pmatrix}n+k-1\\k-1\end{pmatrix} \\ &={\frac 2{n!}}\end{align*} 
Moreover, notice that, when $k>2n+1$, one has \begin{align*}
\begin{pmatrix}n+k\\k\end{pmatrix}^2&={\frac 1{(n!)^2}}\prod_{i=1}^n(k+i)(k+n-i+1)\\ 
\mbox{AM}\geq\mbox{GM}~\Rightarrow\hspace{1cm}&<{\frac 1{(n!)^2}}\prod_{i=1}^n2k^2 
\\ &<{\frac {2^nk^{2n}}{(n!)^2}}\end{align*} 
Equivalently, writing $H_\lambda$ for the eigenspace corresponding to eigenvalue $\lambda$, 
one gets \begin{align*}\lim_{\lambda \rightarrow\infty}
{\frac {\displaystyle\sum_{\lambda\leq a\leq \lambda\sqrt[n]{4}}
\dim H_a}{\lambda^{\frac n2}}}& = {\frac 2{n!}}\end{align*} 
and for all $\lambda>6n^2+3n$, the inequality $\lambda_{ak}<
a^2\lambda_k$ implies \begin{align}\label{later}{\frac {\displaystyle
\sum_{\lambda\leq a\leq \lambda\sqrt[n]{4}}
\dim H_a}{\lambda^{\frac n2}}}& < {\frac {2(2^{\frac n2})}{n!}}\end{align} 
where, the sum ranges over all eigenvalues in $[\lambda,\lambda\sqrt[n]{4}]$. 
Fix a point $x_0\in S^n$. Let $H_t(x)$ be the heat kernel on $S^n$, 
corresponding to Brownian motion started at $x_0$. That is, $H_t(x)$ is the fundamental 
solution to the problem \begin{align*}{\frac{\partial u}{\partial t}}&=-\Delta_{S^n}u\\
\lim_{t\rightarrow 0^+}u(t,x)&=\delta_{x_0}(x)\end{align*} where the convergence is 
taken to be in the $\mbox{weak}^\ast$ topology. A Brownian motion on $S^n$, started at $x_0\in S^n$ 
has infinitesimal generator $H_{\frac t2}(x)$. 
Fixing orthonormal basis $\phi_{k,1},\cdots,\phi_{k,h_k}$ of $H_k:=H_{\lambda_k}$ for each $k\geq 0$, 
one has \begin{align*} H_t(x)&=\sum_{k=0}^\infty e^{-\lambda_kt}
\sum_{i=1}^{h_k}\phi_{k,i}(x_0)\phi_{k,i}(x)\\
&=\sum_{k=0}^\infty e^{-\lambda_kt}h_kP_{k,n}(x\cdot x_0)\end{align*} where the last equality is 
due to \emph{addition theorem} of spherical harmonics (see theorem 2.26 of \cite{Mori}) that 
correspond to the usual addition formula for trigonometric functions when $n=1$; here 
$P_{k,n}(t)$ denotes the Legendre polynomial of degree $k$ and dimension $n+1$; in explicit terms, this polynomial is 
\begin{equation}\label{legendre}P_{k,n}(t)=\displaystyle\sum_{j=0}^{\lfloor {\frac k2}\rfloor}C_{2j}t^{k-2j}
(1-t^2)^j\end{equation} where the coefficients are given by \[C_0=1,\hspace{1cm}
C_{2j}=(-1)^j{\frac{k(k-1)\cdots(k-2j+1)}{(2\cdot4\cdots 2j)(n(n+2)\cdots(n+2j-2))}}\] One has 
\begin{align*}\int_{S^n}H_t(x)^2~d\sigma(x)&=
\sum_{k=0}^\infty e^{-2\lambda_kt}(h_k)^2\int_{S^n}(P_{k,n}(x\cdot x_0))^2~d\sigma(x)\\
&=\sum_{k=0}^\infty e^{-2\lambda_kt}h_k\end{align*} where 
the last equation is a well-known properties of Legendre polynomials 
(see theorem 2.29 of \cite{Mori}). %\begin{equation}\label{vol}\Omega_n:=\mbox{vol}(S^n)=
%{\frac {2\pi^{\frac{n+1}2}}{\Gamma\begin{pmatrix}{\frac {n+1}2} 
%\end{pmatrix}}}\end{equation} 
We note that by theorem 2.29 of \cite{Mori}, 
one has \begin{equation} \label{vol1}
H(t,x,x)=\sum_{k=0}^\infty e^{-\lambda_kt}h_k\end{equation} 
for all $t>0$. It is known that $H_t(x)>0$ for all $t>0$.\\

For $M>0$, let $H_{t,M}(x)$ be defined by \begin{equation}H_{t,M}(x):=\sum_{\lambda_k\leq M} e^{-\lambda_kt}
\sum_{i=1}^{h_k}\phi_{k,i}(x_0)\phi_{k,i}(x)\label{trunc1}\end{equation}

\begin{lem}\label{tail-ineq}
Suppose that $t\in (0,6^{-1})$ and for any $\eta>0$, let $M\geq 4^{\frac{k_0}n}$ 
where \begin{align}\label{bound}k_0>\max\begin{Bmatrix}{\frac 12}\sqrt{\log_2{\frac 1\eta}},
n\log_2\begin{pmatrix} {\frac{n}t}\end{pmatrix}+2n\log_2\log_2\begin{pmatrix} 
{\frac{n}t}\end{pmatrix}\end{Bmatrix};\end{align} 
then the following inequality holds: \[||H_t-H_{t,M}||_{L^2}^2\leq 
\eta^2\]
\end{lem}

\begin{proof}
When $k>n\log_2(3n)$, one has $4^{\frac kn}>6n^2+3n$. For $k\geq 0$, 
write \[I_k=(4^{\frac kn},4^{\frac {k+1}n}]\] For 
$k_0>n\log_2(3n)$, one has \begin{align*}\sum_{\lambda\geq 4^{\frac {k_0}n}}
e^{-2\lambda t}\dim H_\lambda & =\sum_{k\geq k_0}\sum_{\lambda\in I_k}
e^{-2\lambda t}\dim H_\lambda\\ & \leq \sum_{k\geq k_0}\begin{pmatrix}\sup_{\lambda\in I_k}
e^{-2\lambda t}\end{pmatrix}\begin{pmatrix}\sum_{\lambda\in I_k}\dim H_\lambda\end{pmatrix}
\\ & \leq {\frac {2(2^{\frac n2})}{n!}}\sum_{k\geq k_0}2^{1+k}
e^{-(2^{\frac {2k+n}n})t}\end{align*} Suppose that an integer 
$k>k_0$ satisfies \[k\geq n\log_2\begin{pmatrix}{\frac kt}\end{pmatrix}\] 
Then \begin{align*}e^{-(2^{\frac {2k+n}n})t}=(e^{-2^{\frac {2k}n})^{2t}}  
< e^{-\frac{2k^2}t}< 2^{-17k^2}\end{align*} Consider the inequality \begin{align}\label{misc1}
{\frac{k}{\log_2\begin{pmatrix}{\frac kt}\end{pmatrix}}}\geq n\end{align} 
By monotone property of the logarithm function, the following inequality 
is equivalent to (\ref{misc1}) above: \begin{align}\label{final}k\begin{pmatrix}1-
{\frac{\log_2\log_2\begin{pmatrix}{\frac kt}\end{pmatrix}}
{\log_2\begin{pmatrix}{\frac kt}\end{pmatrix}}}\end{pmatrix}\geq n\log_2\begin{pmatrix} 
{\frac nt}\end{pmatrix}\end{align} For $x\in (2^e,+\infty)$, the function 
\[g(x)=1-{\frac{\log_2\log_2x}{\log_2x}}\] satisfies $0<g(x)<1$, has 
global minima $g(2^e)=1-e^{-1}\log_2e>0.46$, and is increasing. 
Since $t\in (0,6^{-1})$ and $n>1$, the condition $n/t>2^e$ is satisfied; because 
$k\geq n\log_2(3n)$, 
the following inequality implies  (\ref{final}): \begin{equation}\label{ult}k\geq 
 n\log_2\begin{pmatrix}{\frac nt}\end{pmatrix}
\begin{pmatrix}1-{\frac{\log_2\log_2\begin{pmatrix}{\frac nt}
\end{pmatrix}}{\log_2\begin{pmatrix}{\frac nt}\end{pmatrix}}}\end{pmatrix}^{-1}\end{equation}
We claim that, for $n>1$ and $t\in (0,6^{-1})$, the following inequality holds:
\begin{align*}{\frac{\log_2\begin{pmatrix}{\frac nt}\end{pmatrix}+2\log_2\log_2
\begin{pmatrix}{\frac nt}\end{pmatrix}}{\log_2\begin{pmatrix}{\frac nt}\end{pmatrix}}}&=1+{\frac{2\log_2\log_2
\begin{pmatrix}{\frac nt}\end{pmatrix}}{\log_2\begin{pmatrix}{\frac nt}\end{pmatrix}}}\\
&\geq \begin{pmatrix}1-
{\frac{\log_2\log_2\begin{pmatrix}{\frac {n}t}\end{pmatrix}}
{\log_2\begin{pmatrix}{\frac{n}t}\end{pmatrix}}}\end{pmatrix}^{-1}\\
&={\frac{\log_2\begin{pmatrix}{\frac nt}\end{pmatrix}} 
{\log_2\begin{pmatrix}{\frac nt}\end{pmatrix}-\log_2\log_2
\begin{pmatrix}{\frac nt}\end{pmatrix}}}\end{align*} 
This is equivalent to \begin{align*}\log_2\begin{pmatrix}{\frac{n}t}\end{pmatrix}&\geq 
2\log_2\log_2\begin{pmatrix}{\frac{n}t}\end{pmatrix}\end{align*} Writing $x={\frac nt}$, this is 
equivalent to $\sqrt x\geq \log_2x$. The function $h(x)=
\sqrt x-\log_2x$ has derivative $h'(x)={\frac 1{2\sqrt x}}-{\frac 1{x\ln 2}}$, which is increasing 
for $x>{\frac 4{(\ln 2)^2}}$, and $h(12)>0$. 
This proves the claim.\\

Thus, $k_0=n\log_2\begin{pmatrix} {\frac{n}t}\end{pmatrix}+2n\log_2\log_2\begin{pmatrix} 
{\frac{n}t}\end{pmatrix}$ implies  \begin{align*} ||H_t-H_{t,M}||_{L^2}^2&=
\sum_{\lambda\geq 4^{\frac {k_0}n}}e^{-2\lambda t}\dim H_\lambda\\ & 
\leq {\frac {2(2^{\frac n2})}{n!}}\sum_{k\geq k_0}2^{1+k}
e^{-(2^{\frac {2k+n}n})t}\\
& \leq {\frac {4(2^{\frac n2})}{n!}}\sum_{k\geq k_0}2^{1+k-17k^2}\\
& \leq {\frac {2^{\frac n2}}{n!}}\sum_{k\geq k_0}2^{-16k^2}\\
& \leq {\frac {2^{{\frac n2}-16k_0^2+1}}{n!}}\\ &\leq \eta^2\end{align*}
%The inequality \begin{align*}{\frac{\Gamma \begin{pmatrix} {\frac{n+1}2}\end{pmatrix}^2}
%{\Gamma(n+1)}}&=\int_0^1(t-t^2)^{\frac{n-1}2}~dt~<4^{-\frac{n-1}2}\end{align*} yields 
%\begin{align*} ||H_t-H_{t,M}||_{L^2}^2&\leq {\frac {\begin{pmatrix} 2\pi
%\end{pmatrix}^{-\frac n2}}{\Gamma \begin{pmatrix} {\frac{n+1}2}\end{pmatrix}}}2^{-16k_0^2}\\
%&\leq {\frac {\begin{pmatrix} 2\pi
%\end{pmatrix}^{-\frac n2}\eta^2}{\Gamma \begin{pmatrix} {\frac{n+1}2}\end{pmatrix}}}\end{align*}
\end{proof}

\begin{rem}\label{remark1}
If $\epsilon\in (0,{\frac 1{3n}})$ and $t=\epsilon^2$, one has \begin{align*}
n\log_2\begin{pmatrix}{\frac nt}\end{pmatrix}+
2n\log_2\log_2\begin{pmatrix}{\frac nt}\end{pmatrix}&<{\frac {3n}2}
\log_2\begin{pmatrix}{\frac 1t}\end{pmatrix}\begin{pmatrix}1+
{\frac {2\log_2\log_2(5n)}
{\log_2(5n)}}\end{pmatrix}\end{align*} We write \[a_n:={\frac {2\log_2\log_2(5n)}
{\log_2(5n)}}\] Then the lemma above implies \[||H_t-H_{t,M}||_{L^2}^2\leq 
\eta^2\] for $M=4^{\frac {k_0}n}$ where $k_0>\max \{\log_2\begin{pmatrix}{\frac 1{\eta}}
\end{pmatrix},{\frac {3n}2}\begin{pmatrix}1+a_n\end{pmatrix}
\log_2\begin{pmatrix}{\frac 1t}\end{pmatrix}\}$.
\end{rem}

\begin{lem}\label{trunc}
Let $c=\min\{\ln\sqrt 2,n^{-1}\}$. For all $t\in (0,c)$ and all $M>0$, following
 inequality holds: \begin{align}\label{trunceq}
||H_{t,M}||^2_{L^2}&<{\frac {t^{-n}}{(n-1)!2^{n-2}}}\end{align}
\end{lem}
\begin{proof}
For any $M>0$, addition theorem implies \[||H_{t,M}||^2_{L^2}:=
\sum_{0<\lambda\leq M}e^{-2\lambda t}\dim H_\lambda\] The function 
$\phi(x)= (x+a)te^{-2xt}$, for $x\in (0,\infty)$, satisfies \[\phi'(x)=te^{-2xt}(1-2(x+a)t),
~~~\phi''(x)=-4t^2e^{-2x}(1-(x+a)t)\] which shows that $\phi(x)\leq (2e)^{-1}e^{2at}$.
 This yields \begin{align*}||H_{t,M}||^2_{L^2}&=
\sum_{0<\lambda_k\leq M}e^{-2k(n+k-1)t}\dim H_k\\ &\leq 
2\sum_{0<\lambda_k\leq M}{\frac 1{(n-1)!}}
e^{-2k(n+k-1)t}\prod_{i=1}^{n-1}(k+i)  \\ &\leq 2
\sum_{0<\lambda_k\leq M}{\frac {e^{-2k^2t}}{(n-1)!t^{n-1}}}
\prod_{i=1}^{n-1}(k+i)te^{-2kt} \\ &\leq 2
\sum_{0<\lambda_k\leq M}{\frac {e^{-2k^2t}}{(n-1)!(2et)^{n-1}}}
\prod_{i=1}^{n-1}e^{2it}\\ &\leq {\frac {t^{-(n-1)}e^{(n-1)(nt-1)}}{(n-1)!2^{n-2}}}
\sum_{k>0}e^{-2k^2t}\\ &< {\frac {t^{-(n-1)}}{(n-1)!2^{n-2}}}
\sum_{k\geq 0}e^{-2kt}\\ &< {\frac {t^{-(n-1)}}
{(1-e^{-2t})2^{n-2}(n-1)!}}\end{align*} 
If $t\in (0,\ln\sqrt2)$ then $1-e^{-2t}>t$, which implies that, for $t\in (0,c)$ 
where $c=\min\{\ln\sqrt 2,n^{-1}\}$, and any $M>0$, one has \[||H_{t,M}||^2_{L^2}
<{\frac {t^{-n}}{(n-1)!2^{n-2}}}\]% Additionally, one has $4n<8^{\frac {n+1}2}\sqrt\pi$, 
%yielding (\ref{trunceq}).
\end{proof}

\vspace{0.5cm}

\section{Hausdorff distance}

We recall the following definition:

\begin{defn}
Given a subset ${\hat S}\subset S^n$, and $\epsilon\geq 0$, let ${\hat S}_\epsilon$ be the union of 
all $\epsilon$-neighbourhoods of points in ${\hat S}$; the Hausdorff distance $d_H({\hat S},S^n)$ is 
defined to be \begin{align*}d_H({\hat S},S^n):&=\max\{\sup_{x\in S^n}\inf_{y\in {\hat S}}d(x,y),
\sup_{y\in {\hat S}}\inf_{x\in S^n}d(x,y)\}\\ &=\inf\{\epsilon\geq 0: 
S^n\subseteq {\hat S}_\epsilon\}\end{align*}
\end{defn}

Our analysis in this section will be based on an application of the following theorem, first appeared in \cite{AW}. 
\begin{thm}[Ahlswede-Winter]\label{tail}
Let $V$ be a finite dimensional Hilbert space, with $\dim V=D$. Let $A_1,
\cdots,A_k$ be independent identically distributed random variables taking 
values in the cone of positive semidefinite operators on $V$, 
such that $\mathbb E[A_i]=A\geq \mu I$ for some $\mu\geq 0$, and 
$A_i\leq I$. Then, for all $\epsilon\in [0,0.5]$, the following holds: 
\begin{equation}\label{taileq}\mathbb P\begin{bmatrix}{\frac 1k}
\sum_{i=1}^kA_i\notin [(1-\epsilon)A,(1+\epsilon)A]\end{bmatrix}
\leq 2D\exp\begin{pmatrix}{\frac{-\epsilon^2\mu k}{2\ln 2}}\end{pmatrix}\end{equation}\\
\end{thm}

Let $S\subset SO_{n+1}$ be a non-empty subset, with $|S|=k$. For $M>9n^2$, let \[E_M:=
\bigoplus_{0<\lambda\leq M}H_\lambda(S^n)\] Recall (inequality \ref{later}) that $\dim E_M
\leq {\frac {2(2M)^{\frac n2}}{n!}}$. Because $\Delta:=\Delta_{S^n}$ is $SO_{n+1}$-invariant,  
the subspace $E_M$ is invariant under the 
operators \begin{equation}\label{ops}
A_s(f)(x):={\frac 12}f(x)+{\frac 14}(f(xs)+f(xs^{-1})),\hspace{1cm}s\in SO_{n+1}
\end{equation} Due to rotation invariance of the surface probability measure $\sigma$, 
the operators $A_s:E_M\rightarrow E_M$ turns out to be self-adjoint. 
Positive semidefiniteness of $A_s$ follows 
from the identity \[\langle A_sf,f\rangle={\frac 14}\int_{S^n}\begin{pmatrix}
f(x)+f(xs)\end{pmatrix}^2~d\sigma(x)\] Moreover, writing 
$\mu$ for the unique right-invariant Haar (probability) measure on $SO_{n+1}$, 
one has \begin{align*}
\begin{pmatrix}\mathbb E_{s\sim\mu}[A_s]\end{pmatrix}f(x)&={\frac 12}f(x)
+{\frac 14}\int_{SO(n+1)}f(xs)~d\mu(s)
+{\frac 14}\int_{SO(n+1)}f(xs^{-1})~d\mu(s)\end{align*} Writing $\tau:
SO_{n+1}\rightarrow S^n$ for the map 
$s\mapsto xs$, one has \begin{align*}\int_{SO(n+1)}f(xs)~d\mu(s)&
=\int_{SO(n+1)}(f\circ\tau)(s)~d\mu(s)\\
&=\int_{S^n}f(y)~d(\tau_\ast \mu)\end{align*} where 
$\tau_\ast \mu(E)=\mu(\tau^{-1}(E))$. Since $\tau_\ast\mu$ is 
rotation-invariant measure on $S^n$, one has $\sigma=\tau_\ast\mu$. 
Because $f\in E_M\subset L^2_0(S^n)$, one has \[\int_{SO(n+1)}f(xs)
~d\mu(s)=\int_{S^n}f(y)~d(\tau_\ast \mu)=0\] Therefore, 
$\begin{pmatrix}\mathbb E_{s\sim\mu}[A_s]\end{pmatrix}f(x)=
{\frac 12}f(x)$, which makes
\begin{equation}\label{expect}\mathbb E_{s\sim\mu}[A_s]=
{\frac 12}I,\hspace{1cm}s\in SO(n+1)\end{equation} Furthermore, 
the operators $I-A_s$ are positive semidefinite for all $s\in SO(n+1)$, because 
\begin{align*}\langle (I-A_s)f,f\rangle 
& ={\frac 14}\int_{S^n}\begin{pmatrix}f(x)-f(xs)\end{pmatrix}^2d\sigma(x).\end{align*}

\begin{thm}\label{eqdist}
Let $S\subset SO_{n+1}$ be a set of order $|S|=k$, chosen independently, 
and uniformly at random from the Haar measure on $SO_{n+1}$ and let ${\hat S}:
=S\sqcup S^{-1}$ be the (multi)set of all elements in $S$ and their inverses. 
Let $\eta>0$ satisfy \[\log_2{\frac 1{\eta}}\geq {\frac{3n}2}(1+a_n)\log_2
\begin{pmatrix}{\frac 1t}\end{pmatrix},\] where $a_n:={\frac {2\log_2\log_2(5n)}
{\log_2(5n)}}$. Let $t\in (0,c)$ 
where $c=\min\{{\frac 16},{\frac 1n}\}$. For \[\delta={\frac {2^{\frac {n+4}2}}{n!\eta}}
\exp\begin{pmatrix}{\frac{-k}{16\ln 2}}\end{pmatrix}\] 
and any integer $l>0$ satisfying $2^l\geq {\frac {t^{-\frac n2}}
{\eta}}$, the following inequality holds: 
\begin{equation}\label{eqdisteq}
\mathbb P\begin{bmatrix}\begin{Vmatrix}1_{S^n}-{\frac 1{(2k)^l}}
\displaystyle\sum_{s\in {\hat S}^l} H_t(x_0s, x)\end{Vmatrix}_{L^2}
\leq 2\eta\end{bmatrix}\geq 1-\delta\end{equation}
\end{thm}

\begin{proof}
Since $t\in (0,c)$ and $c=\min\{{\frac 16},{\frac 1n}\}$, 
one has $\eta^{-\frac 2n}\geq 9n^2$. Setting $\epsilon=0.5$ 
and $M=\eta^{-\frac 2n}$ in (\ref{taileq}) yields \begin{align*}\mathbb P
\begin{bmatrix}{\frac 1k}\sum_{s\in S}A_s\notin [{\frac 14}I,
{\frac 34}I]\end{bmatrix}\leq {\frac {2^{\frac {n+4}2}}{n!\eta}}
\exp\begin{pmatrix}{\frac{-k}{16\ln 2}}\end{pmatrix}\end{align*}
Therefore, for all $f\in E_M$, the following inequality holds: \begin{align*}\mathbb P
\begin{bmatrix}\begin{Vmatrix}{\frac 1k}\sum_{s\in {\hat S}}{\frac 14}f\cdot\tau(s)
\end{Vmatrix}_{L^2}\leq {\frac 14}||f||_{L^2}\end{bmatrix}\geq 1-{\frac {2^{\frac {n+4}2}}{n!\eta}}
\exp\begin{pmatrix}{\frac{-k}{16\ln 2}}\end{pmatrix}\end{align*} 
In particular, writing ${\tilde H}_{t,M}:=1_{S^n}-H_{t,M}$, 
one has \begin{align*}\mathbb P\begin{bmatrix}\begin{Vmatrix}{\frac 1{2k}}
\sum_{s\in {\hat S}}{\tilde H}_{t,M}(xs)\end{Vmatrix}_{L^2}
\leq {\frac 12}||{\tilde H}_{t,M}||_{L^2}\end{bmatrix}\geq 1
-{\frac {2^{\frac {n+4}2}}{n!\eta}}
\exp\begin{pmatrix}{\frac{-k}{16\ln 2}}\end{pmatrix}\end{align*} 
Iterating this inequality $l>0$ times fetches 
\begin{align}\label{iterate}\mathbb P\begin{bmatrix}\begin{Vmatrix}{\frac 1{(2k)^l}}
\sum_{s\in {\hat S}^l}{\tilde H}_{t,M}(xs)
\end{Vmatrix}_{L^2}\leq {\frac 1{2^l}}||{\tilde H}_{t,M}||_{L^2}
\end{bmatrix}\geq 1-{\frac {2^{\frac {n+4}2}}{n!\eta}}
\exp\begin{pmatrix}{\frac{-k}{16\ln 2}}\end{pmatrix}\end{align} Let 
${\tilde H}_t=H_t-1_{S^n}$; then 
\begin{align*}||{\tilde H}_{t,M}||_{L^2}^2&=||{\tilde H}_{t,M}
-{\tilde H}_t||_{L^2}^2+||{\tilde H}_t||_{L^2}^2\\ &=||H_{t,M}
-H_t||_{L^2}^2+||{\tilde H}_t||_{L^2}^2\\ &\leq ||H_{t,M}
-H_t||_{L^2}^2+||H_t||_{L^2}^2\end{align*} 
Hence, using lemma \ref{tail-ineq} and \ref{trunc} in inequality 
(\ref{iterate}), one derives \begin{align*}\mathbb P
\begin{bmatrix}\begin{Vmatrix}{\frac 1{(2k)^l}}\sum_{s\in {\hat S}^l}{\tilde H}_t(xs)
\end{Vmatrix}_{L^2}\leq (2^{-l}t^{-\frac n2}+\eta)
\end{bmatrix}\geq 1-{\frac {2^{\frac {n+4}2}}{n!\eta}}
\exp\begin{pmatrix}{\frac{-k}{16\ln 2}}\end{pmatrix}\end{align*} Since $H_t(xs)=H_t(x_0s^{-1},x)$ and 
${\hat S}$ is inverse-symmetric, this produces (\ref{eqdisteq}).
\end{proof}

The following theorem has appeared in \cite{Nowak-Szarek}:
\begin{thm}[Nowak-Sj\"{o}gren-Szarek]
\begin{comment}Let $M$ be a connected complete Riemannian manifold with 
Riemannian distance $d(.,.)$, and let $H(t,x_0,x)$ be the heat 
kernel corresponding to initial heat source at $x_0$. For any constant $C>4$, 
there exists $C_1>0$ depending on $C, T$, 
the bound $K$ on the sectional curvature of the Riemannian manifold $M$ and $x_0$, so that for all
$t\in (0,T)$, the following inequality holds: \begin{equation}\label{Yaueq}
H(t,x_0,x) \leq C_1(C,T,K, x_0)t^{-\frac {\dim M}2}\exp\begin{pmatrix}
{-\frac{d(x,y)^2}{Ct}}\end{pmatrix}\end{equation}\end{comment}
Let $H(t,x_0,x)$ be the heat kernel on the sphere $S^n$, corresponding to 
Brownian motion initiated at $x_0\in S^n$. Let $n\geq 1$ and fix $T > 0$. 
Let $\phi(x):=\arccos\langle x,x_0\rangle$ be the Riemannian distance, 
so that $\phi(x)\in [0,\pi]$. Then, for all $0<t\leq T$, the inequality 
\[\frac{c}{(t + \pi - \phi)^{\frac{n-1}2}t^{\frac n2}}
\exp\left(- \frac{\phi(x)^2}{4t}\right) \leq H(t,x_0,x) \leq 
\frac{C}{(t + \pi - \phi)^{\frac{n-1}2}t^{\frac n2}}\exp\left(- \frac{\phi(x)^2}{4t}\right)\] 
holds for some constants $c, C >0$ depending only on $n$ and $T$.
\end{thm}

Let $\epsilon,\eta\in (0,1)$ and let $x\in S^n$ be such that $\phi(x)>2\epsilon
\sqrt{\ln{\frac 1{\eta\epsilon^{2n-1}}}}$. Then, taking $T=1$ in the above upperbound, 
one has \begin{align*}H_t(x)&\leq C_{S^n}t^{-n+\frac 12}\exp\begin{pmatrix}
{-\frac{\phi(x)^2}{4\epsilon^2}}\end{pmatrix}\\
&\leq C_{S^n}\eta\end{align*} Notice that the constant 
$C_{S^n}$ is independent of the initial point $x_0\in S^n$. 
Letting $C_n:=1+C_{S^n}$, we have \begin{align}\label{ptwise}
H_{\epsilon^2}(x)<C_n\eta\hspace{0.5cm}&\forall~\epsilon\in (0,1),\nonumber\\
&\forall~x\in S^n~\mbox{s.t.}~r(\epsilon,\eta):=2\epsilon
\sqrt{\ln{\frac 1{\eta\epsilon^{2n-1}}}}<\phi(x)\end{align} 

\begin{lem}\label{punch}
Let $\epsilon\in(0,c)$ where $c=(n+4)^{-1}$. If $r=2\epsilon
\sqrt{\ln {\frac {3C_n}{\epsilon^{2n-1}}}}$ is sufficiently small, then the following inequality 
implies that $x_0{\hat S}^l\subseteq S^n$ is an $r$-net: \begin{equation}\label{neteq}
\begin{Vmatrix}1_{S^n}-
{\frac 1{(2k)^l}}\sum_{s\in{\hat S}^l}H_t(x_0s,x)\end{Vmatrix}_{L^2}\leq {\frac{r^{\frac n2}}3}\sqrt{\frac{
\Gamma\begin{pmatrix}
{\frac{n+1}2}\end{pmatrix}}{\Gamma\begin{pmatrix}
{\frac n2}+1\end{pmatrix}}}\end{equation}
\end{lem}

\begin{proof}Let $0<\eta<(3C_n)^{-1}$; 
then, for any $\epsilon\in (0,1)$, 
and all $x\in S^n$ satisfying the inequality \[d(x,x_0s)>r(\epsilon,\eta),\] it 
follows from (\ref{ptwise}), and positivity of the heat kernel, that $0<H_t(x_0s,x)<{\frac {1_{S^n}}3}$, 
and (hence) \begin{align}\label{need}{\frac {2_{S^n}}3}\leq 1_{S^n}-
H_t(x_0s,x) \leq 1_{S^n}\end{align} Write $B(x_0s,r)\subseteq S^n$ for the Riemannian 
disk of radius $r$, centered at $x_0s\in S^n$. Let $B_n\subseteq \bR^n$ denote 
the unit euclidean ball; one has (see \cite{Gray}) \begin{align}\label{volume}\lim_{r\rightarrow 0}
{\frac{\sigma(B(x_0s,r))}{r^n}}&={\frac{\mbox{vol}(B_n)}{\mbox{vol}(S^n)}}\nonumber\\ 
&={\frac 1{2\sqrt{\pi}}}{\frac{\Gamma\begin{pmatrix}
{\frac{n+1}2}\end{pmatrix}}{\Gamma\begin{pmatrix}
{\frac n2}+1\end{pmatrix}}}\end{align} Here ``vol" denotes the standard 
Lebesgue volume. Now suppose, if possible, that (\ref{neteq}) is satisfied, 
and yet, $x_0{\hat S}^l\subseteq S^n$ is not an $r$-net, so that there 
is $x\in S^n$ such that $d(x,x_0{\hat S}^l)>r$. Writing \[\alpha_n:=\sqrt{\frac{
\Gamma\begin{pmatrix}
{\frac{n+1}2}\end{pmatrix}}{\Gamma\begin{pmatrix}
{\frac n2}+1\end{pmatrix}}},\] we derive from (\ref{neteq}) and (\ref{need}) \begin{align*}{\frac{
r^{\frac n2}\alpha_n}3}&\geq \begin{Vmatrix}
1_{S^n}-{\frac 1{(2k)^l}}\sum_{s\in{\hat S}}H_t(x_0s,x)\end{Vmatrix}_{L^2}\\ &=
\begin{Vmatrix}{\frac 1{(2k)^l}}\sum_{s\in{\hat S}}\begin{pmatrix}1_{S^n}-
H_t(x_0s,x)\end{pmatrix}\end{Vmatrix}_{L^2}\\
&\geq {\frac 23}\begin{pmatrix}\int_{B(x_0s,r)}
d\sigma(x)\end{pmatrix}^{\frac 12}\end{align*} which produces \begin{align*}
{\frac{\sigma(B(x_0s,r))}{r^n}}&={\frac 1{r^n}}\int_{B(x_0s,r)}
d\sigma(x)\\ &\leq {\frac {\alpha_n^2}4}\end{align*} Considering (\ref{volume}), this is impossible 
if $r>0$ is sufficiently small.
%
%Hence, for $r>0$ sufficiently small, one has \begin{align*}
%\begin{Vmatrix}{\tilde H}_t(x_0s,x)
%\end{Vmatrix}_{L^2}&=\begin{Vmatrix}1_{S^n}-
%H_t(x_0s,x)\end{Vmatrix}_{L^2}\\ &\geq {\frac 23}\begin{pmatrix}\int_{B(x_0s,r)}
%d\sigma(x)\end{pmatrix}^{\frac 12}\\ &= {\frac {2r^{\frac n2}\sigma(B)}3}\end{align*} 
%where $B_n\subseteq $
\end{proof}

\begin{thm}\label{main1}
Let $\epsilon\in (0,{\frac 1{3n}})$ be small, and $r=2\epsilon
\sqrt{\ln {\frac {3C_n}{\epsilon^{2n-1}}}}$. Let $S\subset SO_{n+1}$ 
consist of $k$ iid random points, drawn from the Haar measure on $SO_{n+1}$, where 
\[k\geq 8\ln 2\begin{pmatrix}(n+4)+2\ln\begin{pmatrix}{\frac 1{\delta}}\end{pmatrix}
+6n(1+a_n)\ln\begin{pmatrix}{\frac 1{\epsilon}}\end{pmatrix}-\ln (n!)\end{pmatrix},\] 
with $a_n:={\frac {2\log_2\log_2(5n)} {\log_2(5n)}}$. Let $l= 
{\frac n2}\log_2\begin{pmatrix}{\frac 1{r\epsilon}}\end{pmatrix}+(4+3a_n)n\log_2 
\begin{pmatrix}{\frac 1{\epsilon}}\end{pmatrix}$; if $r$ is sufficiently small then the probability  that 
$x_0{\hat S}^l\subseteq S^n$ is an $r$-net in $S^n$  is at least $1-\delta$.\end{thm}

\begin{comment}
$\epsilon<{\frac 1{3n}}$ ensures 
\end{comment}

\begin{proof}
One sees by the remark \ref{remark1} (following lemma \ref{tail-ineq}) 
that for any $\eta>0$, if \[k_0>\max\begin{Bmatrix}
\log_2{\frac 1\eta},\begin{pmatrix}1+a_n\end{pmatrix}{\frac {3n}2}
\log_2\begin{pmatrix}{\frac 1t}\end{pmatrix}\end{Bmatrix}\] and $M=4^{\frac {k_0}n}$, then 
the following inequality holds: \begin{align*}||H_t-H_{t,M}||_{L^2}^2\leq 
\eta^2\end{align*}
Let $\eta:=\epsilon^{3n(1+a_n)}$, so that for sufficiently large $l>0$ (to be determined 
\emph{\`{a} la} theorem \ref{tail}) the parameter 
$M=\eta^{-\frac 2n}$ ensures \begin{comment}because $M\geq 9n^2$\end{comment} (\ref{eqdisteq}) with 
\begin{equation}\label{delta}\delta={\frac {2^{\frac {n+4}2}}{n!\eta}}
\exp\begin{pmatrix}{\frac{-k}{16\ln 2}}\end{pmatrix}\end{equation} 
Taking logarithm of (\ref{delta}), we find that it suffices to take \begin{align*}\label{large}
k\geq 8\ln 2\begin{pmatrix}(n+4)+2\ln\begin{pmatrix}{\frac 1{\delta}}\end{pmatrix}
+2\ln\begin{pmatrix}{\frac 1{\eta}}\end{pmatrix}-\ln(n!)\end{pmatrix}\end{align*} Suppose that $l>0$ is 
large enough so that $2^{-l}t^{-\frac n2}\leq r^{\frac n2}$; for this to be true, we require $l\geq 
{\frac n2}\log_2\begin{pmatrix}{\frac 1{r\epsilon^2}}\end{pmatrix}$. We enforce the inequality 
$2^{-l}t^{-\frac n2}\leq \epsilon^{3n(1+a_n)}$ by requiring \[l\geq (4+3a_n)n\log_2 
\begin{pmatrix}{\frac 1{\epsilon}}\end{pmatrix}.\] Threfore, if $l= (4+3a_n)n\log_2 
\begin{pmatrix}{\frac 1{\epsilon}}\end{pmatrix}+
{\frac n2}\log_2\begin{pmatrix}{\frac 1{r\epsilon}}\end{pmatrix}$, then $2^{-l}
t^{-\frac n2}\leq \min\{r^{\frac n2},\eta\}$ holds. Since $\epsilon>0$ is small, 
one has $\alpha_nr^{\frac n2}>6\eta=6\epsilon^{3n(1+a_n)}$. Thus, by theorem \ref{eqdist}, 
the following inequality holds: \begin{equation}\label{eqdisteq1}
\mathbb P\begin{bmatrix}\begin{Vmatrix}1_{S^n}-{\frac 1{(2k)^l}}
\displaystyle\sum_{s\in {\hat S}^l} H_t(x_0s, x)\end{Vmatrix}_{L^2}
\leq {\frac{r^{\frac n2}\alpha_n}3}\end{bmatrix}\geq 1-\delta.\end{equation} The proof 
is complete by lemma \ref{punch}.\\
\end{proof}

\vspace{1cm}

\section{Equidistribution and Wasserstein Distance}

Let $(Y,d)$ be a compact connected metric space. Let $C(Y)$ be the Banach space of 
continuous functions on $Y$, and ${\mathcal M}(Y)$ its dual --- consisting of linear 
functionals on $C(Y)$ --- equipped with $\mbox{weak}^\ast$ topology; recall that, 
by compactness of $Y$, every linear functional is bounded, and hence, continuous. 
Let ${\mathscr M}(Y)$ be the space of all finite Borel measures on $Y$. By 
\emph{Reisz-Markov theorem}, there is a bijection 
${\mathscr M}(Y)\cong {\mathcal M}(Y)$, defined by \[\mu\mapsto\begin{pmatrix}f\mapsto
\int_Yf~d\mu\end{pmatrix}\] that is closed under addition and scalar multiplication. 
The space ${\mathscr M}(Y)$ inherits the sequential topology on 
${\mathcal M}(Y)$ via this bijection. Thus, one says $\mu_n\Rightarrow\mu$ if and only if 
\[\int_Yf~d\mu_n\rightarrow \int_Yf~d\mu\] for every $f\in C(Y)$. Since the Lipschitz functions 
are dense in $C(Y)$, it suffices to consider only the 1-Lipschitz functions in the above limit.\\

For probability measures $\mu,\nu\in {\mathscr M}(Y)$, the 
\emph{Prokhorov distance} $d_P(\mu,\nu)\geq 0$ is defined 
to be \[d_P(\mu,\nu)=\inf\{\epsilon>0:\mu(B)\leq \nu(B_\epsilon)+\epsilon~\forall~
B\in {\mathscr B}(Y)\}\] where $B_\epsilon:=\{y\in Y:\exists~b\in B,~d(b,y)<\epsilon\}$. 
This gives a metric on the convex subspace ${\mathscr P}(Y)\subset {\mathscr M}(Y)$ of 
probability measures on $Y$, and --- by \emph{Prokhorov's theorem} --- the induced 
metric topology on ${\mathscr P}(Y)$ is the subspace of the weak 
topology on ${\mathscr M}(Y)$; moreover, the space ${\mathscr P}(Y)$ is 
compact.\\

We recall that, in a metric space $(Y,d)$, the 1-\emph{Wasserstein distance} 
between two regular Borel probability measures $\mu$ and $\nu$ on $X$ is defined 
to be \[W_1(\mu,\nu):=\displaystyle\inf_{\lambda\in \Pi(\mu,\nu)}\int_{Y\times Y}d(x,y)~d\lambda\] 
where $\Pi(\mu,\nu)$ is the space of couplings of $\mu$ and $\nu$; that is, $\Pi(\mu,\nu)$ 
is the space of all regular Borel probability measures on $Y\times Y$ 
such that the following holds: $\lambda\in \Pi(\mu,\nu)$ if and only if for every Borel set $B\in 
{\mathscr B}(Y)$, one has \[\lambda(Y\times B)
=\mu(B)\hspace{0.25cm}\mbox{and}\hspace{0.25cm}\lambda(B\times Y)=\nu(B)\] 
Let $\Lip_1(Y)$ be the space of all 1-Lipschitz functions on $Y$; we 
recall that, for any $c>0$, one says $f\in \Lip_c(Y)$ if and only if 
$|f(x)-f(y)|\leq c\cdot d(x,y)$ for all $x,y\in Y$.
The following duality theorem first appeared in \cite{RubK}.

\begin{thm}[Kantorovi\v{c} - Rubin\v{s}te\'{i}n]\label{RubiK}
For any $\mu,\nu\in {\mathscr P}(Y)$, the following equality holds: 
\begin{equation}\label{Rub}W_1(\mu,\nu)=\sup_{\phi\in \Lip_1(Y)}\begin{pmatrix}
\int_Y\phi~d\mu-\int_Y\phi~d\nu\end{pmatrix}\end{equation}
\end{thm}

\begin{defn}\label{equidis}
Let $\epsilon>0$. Let $\mu$ be a Borel probability measure on 
the metric space $(Y,d)$. A finite nonempty subset $U\subset Y$ 
is said to be \emph{strongly} $(\mu,\epsilon)$-\emph{equidistributed} if the 
following inequality holds: \begin{equation}\label{equidin}
\sup_{\phi\in C(Y)}\begin{pmatrix}\int_Y\phi~d\mu-
{\frac 1{|U|}}\sum_{y\in U}\phi(y)\end{pmatrix}<\epsilon||\phi||_{C(Y)}\end{equation}\\
\end{defn}

As mentioned before, for $U\subset Y$ to be strongly $(\mu,\epsilon)$-equidistributed, 
it suffices to have a constant $c:=c(\mu)$ such 
that \begin{equation}\label{equidin}
\sup_{\phi\in \Lip_1(Y)}\begin{pmatrix}\int_Y\phi~d\mu-
{\frac 1{|U|}}\sum_{y\in U}\phi(y)\end{pmatrix}<c\epsilon\end{equation} Below we show 
\emph{strong} $(\mu,\epsilon)$-equidistribution of a subset of $S^n$ of appropriate size 
and low degree of randomness.\\

The following lemma will be useful in course of proving the main theorem of this subsection.

\begin{lem}[Fourier convergence]\label{sob}
Fix $y\in S^n$, and let $0<a<b$; then the Fourier-Laplace expansion of $H_t(y,x)$ converges 
uniformly to $H_t(y,x)$ in $[a,b]\times S^n$.
\end{lem}

\begin{proof}
Fix an orthonormal basis $\phi_{1,k},\cdots,\phi_{h_k,k}$ for the eigenspace $H_k(S^n)$. Then 
the Fourier-Laplace expansion of the heat kernel based at $y\in S^n$ is \[
H_t(y,x)=\sum_{k=0}^\infty e^{-\lambda_kt}\sum_{i=1}^{h_k}\phi_{i,k}(x)\phi_{i,k}(y).\] Write \[
\alpha_k(x):=\sum_{i=1}^{h_k}\phi_{i,k}(x)\phi_{i,k}(y)\] so that \[
H_t(y,x)=\sum_{k=0}^\infty e^{-\lambda_kt}\alpha_k(x).\] One has 
\begin{align}\label{mine}||\alpha_k(x)||_{L^2}^2&=
%\begin{pmatrix}\int_{[a,b]}e^{-2\lambda_kt}~dt\end{pmatrix}
\begin{pmatrix}\int_{S^n}\alpha_k^2(x)~d\sigma(x)\nonumber\end{pmatrix}\\
&=\sum_{i=1}^{h_k}\phi_{i,k}(y)^2\int_{S^n}
\phi_{i,k}(x)^2~d\sigma(x)\nonumber\\ &=\sum_{i=1}^{h_k}\phi_{i,k}(y)^2\nonumber\\ 
&=h_k\end{align} by rotation invariance of the sum 
(see \emph{Lemma 2.19 and 2.29}, \cite{Mori}) \[\sum_{i=1}^{h_k}\phi_{i,k}(y)^2\] Therefore, 
\begin{align*}||\alpha_k(x)||_{C^0(x)}&\leq \sqrt{h_k}
||\alpha_k||_{L^2}\\ &=h_k\end{align*} and since 
\begin{align*}h_k&=\begin{pmatrix}
n+k\\n\end{pmatrix}-\begin{pmatrix}
n+k-2\\n\end{pmatrix}\\ &\leq 2\begin{pmatrix}
n+k-1\\n\end{pmatrix}\\ &\leq 2e^n\begin{pmatrix}
{\frac {n+k-1}n}\end{pmatrix}^n\\ &\leq 2e^{n+k-1},\end{align*} this forces \begin{align*}
\sup_{[a,b]\times S^n}\sum_{k=0}^\infty \begin{vmatrix}
e^{-\lambda_kt}\alpha_k(x)\end{vmatrix}&\leq \sum_{k=0}^\infty 
e^{-\lambda_ka}h_k\\ &\leq \sum_{k=0}^\infty 
2e^{(1-ka)(n+k-1)}\\ &<\infty.\end{align*} The Weierstrass' $M-$test implies uniform convergence of $H_t(x,y)$, 
to a continous function on $[a,b]\times S^n$; the claim follows by uniqueness of the continuous limit.
\end{proof}

\begin{lem}\label{vari}
Let $d(\cdot,\cdot)$ be the metric distance on $S^n$. Let $\sigma$ be the uniform 
surface probability measure on $S^n$. %There is an absolute constant $c>0$ such that 
%for any $t>0$ sufficiently small, 
For all $t > 0$,
one has \begin{align*}\int_{S^n}d(y,x)^2~
H_t(y,x)~d\sigma(x)&\leq nt\end{align*}
\end{lem}

\begin{proof}
Without loss of generality we may assume that $S^n$ is embedded in 
$\mathbb R^{n+1}$ as the unit sphere with 
center at $-{\bf e}_{n+1}=(0,\cdots,0,-1)$, and $y={\bf 0}$. We write $H_t(x):=H_t(0,x)$, 
and let $\sigma_t^\ast$ be the Borel measure 
whose Radon-Nikodym derivative is \[{\frac {d\sigma_t^\ast}{d\sigma}}=H_t(x)\] 
\begin{comment}
Let $\lambda$ be the Lebesgue measure on $\mathbb{R}$. 
When $n = 1$, the heat kernel on $S^1$ can be written as a finite sum, 
$$H_t(\theta) = \sum_{i \in \mathbb{Z}} G_t(\theta + 2\pi i),$$ where $G_t$ is 
the density of a Gaussian with variance $t$ in $\mathbb R$.  Thus,
\begin{align*}\int_{S^1}d(0,x)^2~
H_t(x)~d\sigma(x) & = \int_{(-\pi, \pi]}d(0,\theta)^2~(\sum_{i \in \mathbb{Z}} 
G_t(\theta + 2\pi i) )d\sigma(\theta)\\
& < \int_{(-\infty, \infty]}d(0,x)^2~G_t(\theta + 2\pi i) d\lambda(x)\\
& = t.\end{align*} This proves the result for $n = 1$. \end{comment}
Let $\{X_u\mid u\in [0,t]\}$ be a standard Brownian motion on $S^n$ with infinitesimal 
generator $H_{\frac t2}(x)$. For each positive integer $m>0$, consider the equi-partition 
\[0=t_0<t_1<\cdots<t_m=t\] where $t_{i+1}-t_i=m^{-1}t$ for $i=0,1,\cdots,m-1$. 
Now define $\{X_i^{(m)}\}_{i=0}^m$ as follows: \[X_i^{(m)}=X_{\frac {it}m}\] 
By linearity of expectation, for any integer $m>0$ 
one has \begin{align}\label{var}{\mathbb E}(||X_m^{(m)}||^2)=~&\int_{S^n}||x||^2~H_t(x)
~d\sigma(x)\nonumber\\ =~&{\mathbb E}(||X_{m-1}^{(m)}||^2)+2{\mathbb E}(\langle 
X_m^{(m)}-X_{m-1}^{(m)},X_{n-1}^{(m)}\rangle)
+{\mathbb E}(||X_m^{(m)}-X_{m-1}^{(m)}||^2)\nonumber\\
=~&2\sum_{i=1}^m{\mathbb E}(\langle 
X_i^{(m)}-X_{i-1}^{(m)},X_{i-1}^{(m)}\rangle)+\sum_{i=1}^m
{\mathbb E}(||X_i^{(m)}-X_{i-1}^{(m)}||^2)\end{align} 
Fix a realization of the Brownian motion $\{X_u\mid u\in [0,t]\}$. 
For integer $1\leq i\leq m$, we consider the tangent space 
$T_{X_{i-1}^{(m)}}(S^n)$. Write \[Z_{i-1}^{(m)}:=
\mbox{argmin}_{z\in T_{X_{i-1}^{(m)}}(S^n)}||z||,
\hspace{0.5cm}Y_i^{(m)}:=\mbox{argmin}_{z\in T_{X_{i-1}^{(m)}}
(S^n)}||z-X_i^{(m)}||\] In explicit terms, one has 
\begin{align*}Y_i^{(m)}&=X_i^{(m)}-\langle {\bf e}_{n+1}+X_{i-1}^{(m)},
X_i^{(m)}-X_{i-1}^{(m)}\rangle ({\bf e}_{n+1}+X_{i-1}^{(m)})\\
Z_{i-1}^{(m)}&=\langle {\bf e}_{n+1}+X_{i-1}^{(m)},X_{i-1}^{(m)}
\rangle ({\bf e}_{n+1}+X_{i-1}^{(m)})\end{align*} 

Orthogonality relations such as 
\[Y_i^{(m)}-X_{i-1}^{(m)}~\bot~Z_{i-1}^{(m)},\hspace{1cm}\mbox{and}\hspace{
1cm}X_i^{(m)}-Y_i^{(m)}~\bot~X_{i-1}^{(m)}-Z_{i-1}^{(m)}\] are immediate; moreover, one has 
\begin{align*}\langle X_i^{(m)}-Y_i^{(m)},
Z_{i-1}^{(m)}\rangle &=\langle {\bf e}_{n+1}+X_{i+1}^{(m)},
X_{i-1}^{(m)}\rangle \langle {\bf e}_{n+1}+X_{i+1}^{(m)},
X_i^{(m)}-X_{i-1}^{(m)}\rangle \\ &=\langle {\bf n},
X_{i-1}^{(m)}\rangle \langle {\bf n},
X_i^{(m)}-X_{i-1}^{(m)}\rangle\end{align*} where ${\bf n}=
-{\bf e}_{n+1}-X_{i+1}^{(m)}$ is the unit normal to 
$T_{X_{i-1}^{(m)}(S^n)}$ pointing inward. From the inequalities 
%follows from euclidean geometry applied to great circles passing through, 
%respectively, (1) ${\bf 0}$ and $X_{i-1}^{(n)}$, and (2) $X_i^{(n)}$ and $X_{i-1}^{(n)}$
\[\langle {\bf n},X_{i-1}^{(m)}\rangle\leq 0\leq \langle {\bf n},
X_i^{(m)}-X_{i-1}^{(m)}\rangle,\] one has $\langle X_i^{(m)}-Y_i^{(m)},
Z_{i-1}^{(m)}\rangle\leq 0$. Hence, 
\begin{align*}\langle X_i^{(m)}
-X_{i-1}^{(m)},X_{i-1}^{(m)}\rangle&=
\langle Y_i^{(m)}-X_{i-1}^{(m)},Z_{i-1}^{(m)}\rangle+\langle 
X_i^{(m)}-Y_i^{(m)},X_{i-1}^{(m)}-Z_{i-1}^{(m)}\rangle\\ 
&\hspace{1cm}+\langle X_i^{(m)}-Y_i^{(m)},Z_{i-1}^{(m)}\rangle
+\langle Y_i^{(m)}-X_{i-1}^{(m)},X_{i-1}^{(m)}-Z_{i-1}^{(m)}\rangle\\ & \leq 
\langle Y_i^{(m)}-X_{i-1}^{(m)},X_{i-1}^{(m)}-Z_{i-1}^{(m)}\rangle\end{align*} 
Suppose $X_{i-1}^{(m)}=z$ and $X_i=z'$ in $S^n$. 
Let $z'':=z''(z,z')\in S^n$ be such that \begin{equation}\label{comp}
z''+z'-2z=\langle {\bf n},z''+z'-2z\rangle {\bf n}.\end{equation} 
%Since \[|\langle {\bf n},z''+z'-2z\rangle|=|\langle {\bf n},z''+z'-2z\rangle
%|\cdot ||z''+z'-2z||\] 
Since the function \[g(z''):=||z''+z'-2z-\langle {\bf n},z''+z'-2z\rangle{\bf n}||\] 
takes arbitrarily small positive values, 
such a point $z''\in S^n$ --- that satisfies (\ref{comp}) --- exists by continuity of $g(z'')$ 
and compactness of $S^n$.
Note that \begin{align*}&{\mathbb P}\{Y_i^{(m)}-X_{i-1}^{(m)}=z''-z
-\langle {\bf n},z''-z\rangle {\bf n} \mid X_{i-1}^{(m)}=z\}\\
=~&{\mathbb P}\{Y_i^{(m)}-X_{i-1}^{(m)}=
z'-z-\langle {\bf n},z'-z\rangle {\bf n} \mid X_{i-1}^{(m)}=z\}\end{align*} 
by independence of increments for Brownian motion on euclidean space. 
From \begin{align*}&\langle Y_i^{(m)}-X_{i-1}^{(m)},
X_{i-1}^{(m)}-Z_{i-1}^{(m)}\rangle\mid_{X_{i-1}^{(m)}=z, X_i^{(m)}=z'}\\=~&
-\langle Y_i^{(m)}-X_{i-1}^{(m)},
X_{i-1}^{(m)}-Z_{i-1}^{(m)}\rangle\mid_{X_{i-1}^{(m)}=z, X_i^{(m)}=z''}\end{align*} 
we derive \begin{align*}\mathbb E(\langle Y_i^{(m)}-X_{i-1}^{(m)},
X_{i-1}^{(m)}-Z_{i-1}^{(m)}\rangle)&=0\end{align*}

It thus suffices to prove 
\begin{lem}\label{stereo}
$$\lim_{m\rightarrow\infty}\sum_{i=1}^m
{\mathbb E}(||X_i^{(m)}-X_{i-1}^{(m)}||^2) = nt.$$
\end{lem}
\begin{proof}
We will use the stereographic projection of $S^n\setminus \{0\}$ onto $\mathbb{R}^n$.
It can be shown (see for example \cite{Carne}) that the image $Y_t$ of a standard Brownian 
motion on $S^n$ via the stereographic projection onto $\mathbb{R}^n$, 
where $r = |Y_t|$, with $\Delta_{\mathbb{R}^n}$ being the Laplacian 
on $\mathbb{R}^n$, has an infinitesimal generator $(1/2) \Delta_{S^n}$ that satisfies

$$\Delta_{S^n} = \left(\frac{1 + r^2}{2}\right)^2  \Delta_{{\mathbb R}^n} 
- (n-2) \left(\frac{r(1 + r^2)}{2}\right)\frac{\partial}{\partial r}.$$
Applying this to the function $f(x) = \|x\|^2$, we see that 

\begin{eqnarray*} \lim_{t \rightarrow 0} {\mathbb E^0} Y_t^2/t 
& = &  (1/2)\Delta_{S^n} r^2|_{r=0}\\
&  = &  (1/2)  \left(\frac{1 + r^2}{2}\right)^2  \Delta_{{\mathbb R}^n} 
(r^2)|_{r = 0} - (n-2) \left(\frac{r(1 + r^2)}{4}\right)\frac{\partial}{\partial r}(r^2)|_{r = 0}\\
& = & n.\end{eqnarray*}

It follows that for any $i$, $$m {\mathbb E}(||X_i^{(m)}-X_{i-1}^{(m)}||^2),$$ converges as $m \rightarrow \infty$ to 
$nt$,  proving the lemma.
\end{proof}

\begin{comment}
By an application of Holder's inequality we get 
\begin{align*}\int_{S^n}d(y,x)~
H_t(y,x)~d\sigma(x)&\leq \Omega_n^{\frac 12}\begin{pmatrix}\int_{S^n}d(y,x)^2~
H_t(y,x)~d\sigma(x)\end{pmatrix}^{\frac 12}\\
&=\Omega_n^{\frac 12} \sqrt t\end{align*} for all $0<t<\tau$. 
Since $\Omega_n\rightarrow 0$ as 
$n\rightarrow \infty$, the proof is complete.
\end{comment}
\end{proof}

\begin{cor}
Let $n>1$ and for integers $k\geq 0$, let $P_{k,n}(t)$ be the \textit{Legendre} 
polynomial of degree $k$ and dimension $n+1$. Let 
$h_k=\dim H_k(S^n)$ and \[\gamma_k
=\int_{-1}^1(1-t)^{\frac 12}(1-t^2)^{\frac{n-2}2}P_{k,n}(t)~dt\] Then the following inequality 
holds for all $n\geq 4$: \begin{align}\label{nasty}\sum_{k=0}^\infty e^{-\lambda_kt}
\gamma_kh_k\leq \sqrt {nt}.\end{align}
\end{cor}

\begin{proof}
We recall the \emph{Hecke-Funk formula}: for any function $\chi:[-1,1]\rightarrow\bR$, 
which satisfies the inequality \[\int_{[-1,1]}|\chi(t)|(1-t^2)^{\frac{n-2}2}
~dt<\infty,\] and any eigenfunction $\phi\in {\mathcal H}_k(S^n)$ and point $y\in S^n$ 
one has \begin{align*}\int_{S^n} \chi(y\cdot x)\phi(x)~d\sigma(x)&
={\frac{\Omega_{n-1}}{\Omega_n}}\phi(y)
\int_{-1}^1\chi(t)(1-t^2)^{\frac{n-2}2}P_k^n(t)~dt\end{align*} Consider the  
function $\chi(t)=\sqrt 2(1-t)^{\frac 12}$, taking values in $[0,2]$; 
this satisfies the hypothesis in Hecke-Funk formula, and 
since $d(y,x)=\sqrt 2(1-\langle y,x\rangle)^{\frac 12}$, one gets 
\begin{align*}\int_{S^n} d(y, x)\phi(x)~d\sigma(x)&
={\frac{\sqrt 2\Omega_{n-1}}{\Omega_n}}\phi(y)
\int_{-1}^1(1-t)^{\frac 12}(1-t^2)^{\frac{n-2}2}P_{k,n}(t)~dt\end{align*} Consider the Fourier-
Laplace expansion of the heat kernel, as in lemma \ref{sob} above. By uniform 
convergence (lemma \ref{sob}) of the Fourier-Laplace expansion of 
the heat kernel, one has \begin{align*}\int_{S^n}d(y,x)~H_t(y,x)~d\sigma(x)
&=\sum_{k=0}^\infty e^{-\lambda_kt}\begin{pmatrix}\sum_{i=1}^{h_k}\phi_{i,k}(y)
\int_{S^n}d(y,x)~\phi_{i,k}(x)~d\sigma(x)\end{pmatrix}\\ &= 
{\frac{\sqrt 2\Omega_{n-1}}{\Omega_n}}\sum_{k=0}^\infty e^{-\lambda_kt}
\gamma_k\begin{pmatrix}\sum_{i=1}^{h_k}\phi_{i,k}^2(y)
\end{pmatrix}\\ &= {\frac{\sqrt 2\Omega_{n-1}}{\Omega_n}}\sum_{k=0}^\infty e^{-\lambda_kt}
\gamma_kh_k\end{align*} Note that \begin{align*}{\frac{\sqrt 2\Omega_{n-1}}
{\Omega_n}}&=\sqrt {\frac 2{\pi}}{\frac{\Gamma \begin{pmatrix}{\frac{n+1}2}
\end{pmatrix}}{\Gamma \begin{pmatrix}{\frac n2}
\end{pmatrix}}}\\ &\geq 1\end{align*} for all $n\geq 4$. Hence, Lemma \ref{vari} together with 
H\"{o}lder inequality implies 
\begin{align*}\sqrt{nt}&\geq \int_{S^n}d(y,x)~
H_t(y,x)~d\sigma(x)\\ &\geq \sum_{k=0}^\infty e^{-\lambda_kt}
\gamma_kh_k\end{align*}
\end{proof}

\begin{thm}\label{Wassers}
For $n>1$, let $\mu$ be the probability measure on $S^n$ corresponding to $\sigma$.
Let $\epsilon,\delta>0$ be sufficiently small and $r=2\epsilon
\sqrt{\ln {\frac {3C_n}{\epsilon^{2n-1}}}}$. Let $S\subseteq SO(n+1)$ be a random 
subset such that $|S|=k$ satisfies the inequality in theorem \ref{main1}, namely 
\[k>8\ln 2\begin{pmatrix}(n+4)+2\ln\begin{pmatrix}{\frac 1{\delta}}\end{pmatrix}
+6n(1+a_n)\ln\begin{pmatrix}{\frac 1{\epsilon}}\end{pmatrix}-\ln (n!)\end{pmatrix},\] 
where $a_n:={\frac {2\log_2\log_2(5n)}
{\log_2(5n)}}$.
Let $x_0\in S^n$ and let $\nu$ be the uniform probability measure on $S^n$, 
supported on ${\hat S}^lx_0$, where ${\hat S} = S \cup S^{-1}$ as before and \[l={\frac n2}\log_2
\begin{pmatrix}{\frac 1{r\epsilon}}\end{pmatrix}+(4+3a_n)\log_2
\begin{pmatrix}{\frac 1{\epsilon}}\end{pmatrix}.\] Then, with 
probability at least $1-\delta$, the 
following inequality holds: \[W_1(\sigma,\nu)\leq \epsilon\]
\end{thm}

\begin{proof}
%It suffices to show that there is $\lambda\in\Pi(\sigma,\nu)$ such that \[\int_{S^n \times S^n}d(x,y)~d\lambda<\epsilon\] Define Borel measure $\lambda$ on $S^n\times S^n$ by \[\lambda(A_1\times A_2)=\int_{A_1\times A_2}\sigma(x)\nu(y)~d\sigma(x,y)\] One has $\lambda \in \Pi(\mu,\nu)$ since \begin{align*}\lambda(A_1\times S^n)&={\frac 1{\Omega_n}}\int_{A_1\times S^n}\nu(y)~d\sigma(x,y)\\&=\mu(A_1)\\\lambda(S^n\times A_2)&={\frac 1{\Omega_n}}\int_{S^n\times A_2}\nu(y)~d\sigma(x,y)\\&=\nu(A_2)\end{align*} Now, \begin{align*}\int_{S^n\times S^n}d(x,y)~d\lambda&=\int_{S^n}\begin{pmatrix}\int_{S^n}d(x,y)~d\sigma(x)\end{pmatrix}d\nu(y)\\&=\nu(A_2)\end{align*} For any $x\in S^n$, one has \begin{align*}\int_{S^n}d(x,y)~d\sigma(x)&=\int_{S^{n-1}}d(x,y)~d\sigma(x)\end{align*}

Let $\Lip_{1,0}(S^n)$ be the set of mean-zero $\Lip_1$-functions on $S^n$. 
By theorem \ref{RubiK}, it suffices to show that 
\begin{equation}\sup_{\phi\in \Lip_{1,0}(S^n)}\begin{pmatrix}
\int_Y\phi~d\sigma-\int_Y\phi~d\nu\end{pmatrix}<\epsilon\end{equation} 
For any such function $\phi\in \Lip_{1,0}(S^n)$, if $\phi(x_0)=||\phi||_{L^\infty}$ then 
\begin{align*}0&=\int_{S^n}\phi(x)~d\mu(x) \\ &=
\int_{S^n}\phi(x_0)~d\mu(x) + \int_{S^n}\begin{pmatrix}\phi(x)-\phi(x_0)\end{pmatrix}
~d\mu(x) \\ &=\phi(x_0)+\int_{S^n}\begin{pmatrix}\phi(x)-\phi(x_0)\end{pmatrix}
~d\mu(x)\\ \Rightarrow\hspace{0.5cm}\phi(x_0)&\leq \int_{S^n}|\phi(x)-\phi(x_0)|
~d\mu(x)\\ &\leq \int_{S^n}d(x,x_0)~d\mu(x)\\
%\hspace{3.25cm}\mbox{recall:}~d(x,x_0)=\arccos\langle x,x_0\rangle\\
&\leq 2\end{align*}
For sufficiently small $t>0$, we let $\nu_t^\ast$ be the Borel probability measure on $S^n$ whose density 
is \[{\frac {d\nu^\ast_t(x)}{d\sigma}}={\frac 1{|{\hat S}^l|}}\sum_{y
\in {\hat S}^lx_0}H_t(y,x)\] Then, for $t={\epsilon^2}$, one has \begin{align}W_1(\sigma,\nu^\ast_t)& =  \sup_{\phi\in\Lip_{1,0}(S^n) }
\begin{vmatrix}\int_{S^n}\phi(x)~d\sigma(x)-
\int_{S^n}\phi(x)~d\nu^\ast_t(x)\end{vmatrix}\nonumber\\ &\leq \sup_{\phi\in\Lip_{1,0}(S^n)}\int_{S^n}|\phi(x)|\cdot 
\begin{vmatrix}1_{S^n}-{\frac 1{(2k)^l}}\sum_{y
\in {\hat S}^lx_0}H_t(y,x)\end{vmatrix}d\sigma(x)\nonumber\\
\hspace{1cm}&\leq \sup_{\phi\in\Lip_{1,0}(S^n) }
\begin{Vmatrix}\phi\end{Vmatrix}_{L^\infty}\cdot
\int_{S^n}\begin{vmatrix}1_{S^n}-{\frac 1{(2k)^l}}\sum_{y
\in {\hat S}^lx_0}H_t(y,x)\end{vmatrix}d\sigma(x)\nonumber\\
\label{1stw}&\leq 2\epsilon^{3n}
\hspace{3cm}(\mbox{see Theorem}~\ref{main1})\end{align} with probability at least $1-\delta$.\\

For any function $\phi\in\Lip_{1,0}(S^n)$, define ${\tilde {\phi}}_t:S^n\rightarrow\bR$ to be 
\[{\tilde {\phi}}_t(x)={\frac 1{|{\hat S}^l|}}\sum_{y
\in {\hat S}^lx_0}\phi(y)H_t(y,x)\] From uniform convergence 
of the Fourier-Laplace expansion of heat-kernel, it 
follows that $\int_{S^n}H_t(y,x)~d\sigma(x)=1$; hence, putting $t=\epsilon^2$, one has 
\begin{align}\label{secondl}\int_{S^n}{\tilde{\phi}}_{\epsilon^2}(x)~d\sigma(x) & = 
{\frac 1{|{\hat S}^l|}}\sum_{y\in {\hat S}^lx_0}\phi(y)\int_{S^n}H_t(y,x)~d\sigma(x)
%\begin{pmatrix}{\frac 1{|{\hat S}^l|}}\sum_{y\in {\hat S}^lx_0}H_{\epsilon^2}(x,y)\end{pmatrix}
\nonumber\\ &=\int_{S^n}\phi(x)~d\nu(x).\end{align} Moreover, \begin{align}\label{last}
&~\begin{vmatrix}\int_{S^n}\phi(x)~d\nu^\ast_t(x)-\int_{S^n}
{\tilde{\phi}}_t(x)~d\sigma(x)\end{vmatrix}\nonumber\\ =&~{\frac 1{(2k)^l}}
\begin{vmatrix}\sum_{y\in {\hat S}^lx_0}\int_{S^n}\begin{pmatrix}\phi(x)-
\phi(y)\end{pmatrix}H_t(y,x)~d\sigma(x)\end{vmatrix}\nonumber\\ 
\leq&~ {\frac 1{(2k)^l}}\sum_{y\in {\hat S}^lx_0}\int_{S^n}\begin{vmatrix}\phi(x)-
\phi(y)\end{vmatrix}H_t(y,x)~d\sigma(x)\nonumber\\ \leq&~ 
{\frac 1{(2k)^l}}\sum_{y\in {\hat S}^lx_0}\int_{S^n}d(y,x)~H_t(y,x)~
d\sigma(x)\nonumber\\ \leq~& {\frac {n\sqrt t}{(2k)^l}}\end{align} 
by lemma \ref{vari} and H\"{o}lder inequality applied to $d(y,x)=d(y,x)\cdot 1_{S^n}$. 
%Here $d(x,y)=\arccos\langle x,y\rangle$ is the Riemannian distance on $S^n$. 
Therefore, for $t=\epsilon^2>0$ sufficiently small, equations (\ref{1stw}), (\ref{secondl}), 
and (\ref{last}) yield \begin{align*}W_1(\mu,\nu)&\leq 
W_1(\mu,\nu^\ast_t)+W_1(\nu^\ast_t,\nu)\\ &\leq \epsilon\end{align*}
\end{proof}

\vspace{0.5cm}

\section{Conclusion}
We proved two results about the finite time behavior of a random Markov Chain 
on the sphere $S^n$ whose transitions correspond to rotations chosen uniformly 
at random. The first result states that for $k = O(n\ln\frac{1}{\epsilon} + 
\frac{1}{\delta})$ random rotations and $\ell = {O}(n \ln 1/\epsilon)$, if one 
takes the image of the north pole on the sphere under all possible words of 
length $\ell$ in the $k$ alphabets and their inverses, one obtains an $\epsilon-$net 
with high probability. For these parameters, the value of $(2k)^\ell$ is close to 
the volumetric lower bound of $(1/\epsilon)^{\Omega(n)}$ on the size of an 
$\epsilon-$net of $S^n$. Secondly, we show that this $\epsilon-$net is 
equidistributed with probability at least $1 - \delta$  in the sense that the 
$1-$Wasserstein distance of the uniform measure on the net is within 
$\epsilon$ of the uniform measure on $S^n$.\\

These results can respectively be applied to approximately minimize a $1-$
Lipschitz function on the sphere (by evaluation on the $\epsilon-$net) and in 
to approximately integrate a $1-$Lipschitz function on the sphere. In both cases 
the approximation is within an additive $\epsilon$ of the true value.

%%%%%%%%%%%%%%%%%%%%%%%%%%%%%%%%%%%%%%%%%%%%%%%%%%%%%%%%%%%%%%%%%%%%%%%%
% Bibliography
%%%%%%%%%%%%%%%%%%%%%%%%%%%%%%%%%%%%%%%%%%%%%%%%%%%%%%%%%%%%%%%%%%%%%%%%
\newpage
\bibliographystyle{amsplain}
%\bibliography{topology}
\def\noopsort#1{}\def\MR#1{}
\providecommand{\bysame}{\leavevmode\hbox to3em{\hrulefill}\thinspace}
\providecommand{\MR}{\relax\ifhmode\unskip\space\fi MR }
\providecommand{\MRhref}[2]{%
  \href{http://www.ams.org/mathscinet-getitem?mr=#1}{#2}
}
\providecommand{\href}[2]{#2}

\end{document}